\documentclass[twoside,leqno,british,10pt]
{amsart}
\usepackage{amsmath,amssymb,amsbsy,amsthm}
\usepackage{eucal,url,amssymb,
enumerate,amscd,stmaryrd
}
\usepackage{verbatim}
\usepackage[T2A]{fontenc}

\usepackage[british]{babel}
\usepackage[pagebackref,colorlinks=true]{hyperref}  

\begin{document}

\begin{abstract}
We extract a new class of paracontact paracomplex Riemannian manifolds  arising from certain cone construction, call  it para-Sasaki-like Riemannian manifold and give explicit examples. 
We define a hyperbolic extension of a paraholomorphic paracomplex Riemannian manifold, which is a local product of two Riemannian spaces with equal dimensions, showing that it is a para-Sasaki-like Riemannian manifold. If the starting paraholomorphic paracomplex Riemannian manifold is complete Einstein with negative scalar curvature then its hyperbolic extension is a complete Einstein para-Sasaki-like Riemannian manifold with negative scalar curvature thus producing new examples of complete Einstein Riemannian manifold with negative scalar curvature.
\end{abstract}


\title[Para-Sasaki-like Riemannian manifolds and new Einstein metrics]
{Para-Sasaki-like  Riemannian manifolds and new Einstein metrics}

\author[S. Ivanov]{Stefan Ivanov}
\address[S. Ivanov]{Department of Geometry, Faculty of Mathematics and Informatics,
St. Kliment Ohridski University of Sofia, 5 James Bourchier Blvd,
1164 Sofia, Bulgaria \&
Institute of Mathematics and Informatics, Bulgarian Academy of Sciences, Bulgaria}
\email{ivanovsp@fmi.uni-sofia.bg}
\author[H. Manev]{Hristo Manev}
\address[H. Manev]{Department of Medical Informatics, Biostatistics and E-Learning, Faculty of Public
Health, Medical University of Plovdiv, 15A 	Vasil Aprilov Blvd, 4002 Plovdiv,
Bulgaria}
\email{hristo.manev@mu-plovdiv.bg}
\author[M. Manev]{Mancho Manev}
\address[M. Manev]{Department of Algebra and Geometry, Faculty of Mathematics and
Informatics, University of Plovdiv Paisii Hilendarski, 24 Tzar Asen St, 4000 Plovdiv,
Bulgaria \&
Department of Medical Informatics, Biostatistics and E-Learning, Faculty of Public
Health, Medical University of Plovdiv, 15A 	Vasil Aprilov Blvd, 4002 Plovdiv,
Bulgaria}
\email{mmanev@uni-plovdiv.bg}



\frenchspacing

\newcommand{\A}{\allowbreak{}}
\newcommand{\ie}{\textit{i.e.}, }
\newcommand{\X}{\mathfrak{X}}
\newcommand{\W}{\mathcal{W}}
\newcommand{\F}{\mathcal{F}}
\newcommand{\T}{\mathcal{T}}
\newcommand{\TT}{\mathfrak{T}}
\newcommand{\M}{(M,\allowbreak{}\ff,\allowbreak{}\xi,\allowbreak{}\eta,\allowbreak{}g)}
\newcommand{\Lf}{(G,\ff,\xi,\eta,g)}
\newcommand{\R}{\mathbb{R}}
\newcommand{\CC}{\mathbb{C}}
\newcommand{\s}{\mathfrak{S}}
\newcommand{\n}{\nabla}
\newcommand{\ff}{\phi}
\newcommand{\id}{{\rm id}}
\newcommand{\tr}{{\rm tr}}
\newcommand{\al}{\alpha}
\newcommand{\bt}{\beta}
\newcommand{\gm}{\gamma}
\newcommand{\lm}{\lambda}
\newcommand{\om}{\omega}
\newcommand{\ta}{\theta}
\newcommand{\ea}{\varepsilon_\alpha}
\newcommand{\eb}{\varepsilon_\beta}
\newcommand{\eg}{\varepsilon_\gamma}
\newcommand{\sx}{\mathop{\mathfrak{S}}\limits_{x,y,z}}
\newcommand{\norm}[1]{\left\Vert#1\right\Vert ^2}
\newcommand{\nf}{\norm{\n \ff}}
\newcommand{\Span}{\mathrm{span}}
\newcommand{\D}{\mathrm{d}}
\newcommand{\ddr}{\tfrac{\D}{\D r}}
\newcommand{\g}{\check{g}}
\newcommand{\nn}{\check{\n}}
\newcommand{\cP}{\check{P}}
\newcommand{\chh}{\check{h}}
\newcommand{\cwh}{\widetilde{\check{h}}}
\newcommand{\wh}{\widetilde{h}}
\newcommand{\wg}{\widetilde{g}}

\newcommand{\thmref}[1]{The\-o\-rem~\ref{#1}}
\newcommand{\propref}[1]{Pro\-po\-si\-ti\-on~\ref{#1}}
\newcommand{\secref}[1]{\S\ref{#1}}
\newcommand{\lemref}[1]{Lem\-ma~\ref{#1}}
\newcommand{\dfnref}[1]{De\-fi\-ni\-ti\-on~\ref{#1}}
\newcommand{\corref}[1]{Corollary~\ref{#1}}



\newtheorem{thm}{Theorem}[section]
\newtheorem{lem}[thm]{Lemma}
\newtheorem{prop}[thm]{Proposition}
\newtheorem{cor}[thm]{Corollary}
\newtheorem{corr}[thm]{Remark}

\theoremstyle{definition}
\newtheorem{defn}{Definition}[section]
\newtheorem{conv}{Convention}

\hyphenation{Her-mi-ti-an man-i-fold ah-ler-ian pa-ra-con-tact
Rie-mann-ian para-com-plex}




\begin{center}
\date{\today}
\end{center}

\maketitle


\setcounter{tocdepth}{3}
 \tableofcontents


\section*{Introduction}


In 1976 I. Sato \cite{Sato76} introduced the concept of almost
paracontact Riemannian manifolds as analogue of almost contact
Riemannian manifolds  \cite{Blair02,Sas60}. Later on, in 1980 S. Sasaki \cite{Sas80}
defined the notion of an almost paracontact Riemannian manifold of
type $(p, q)$, where $p$ and $q$ are the
numbers of the multiplicity of the structure eigenvalues $1$ and
$-1$, respectively. In addition, there is a simple eigenvalue 0.

In this paper we consider a $(2n+1)$-dimensional almost paracontact Riemannian manifolds
of type $(n, n)$, \ie  $p = q = n$ 
and the paracontact distribution  can be considered
as a $2n$-dimensional
almost paracomplex Riemannian distribution 
with almost paracomplex 
structure
and a structure group $O(n)\times O(n)$.
The paracomplex geometry has been studied since the
first papers by P.\,K. Rashevskij \cite{Rash48}, P. Libermann \cite{Lib52}, and E.\,M. Patterson
\cite{Patt54} until now, from several different points of view.
In particular, the almost paracomplex Riemannian manifolds are classified by M. Staikova and K. Gribachev in
\cite{StaGri92}.

We call these $(2n+1)$-dimensional manifolds
\emph{almost paracontact  paracomplex Riemannian manifolds} (or briefly \emph{apcpcR manifolds}). 
A natural example is the  direct product of an almost paracomplex Riemannian manifold 
with the real line. Accordingly, any real hypersurface of an almost paracomplex Riemannian manifold admits an almost paracontact  paracomplex Riemannian structure.

An odd-dimensional manifold  $(M,\phi,\xi,\eta)$
is said to be an \emph{almost paracontact manifold}
if $\phi$ is a $(1, 1)$-tensor field, 
$\xi$ is a vector field and $\eta$ is a 1-form,
which satisfy the following conditions:
\begin{equation}\label{str}
\phi^2=\id-\eta\otimes\xi,\quad \eta(\xi)=1 \quad  \text{consequently} \quad \ff\xi = 0,\quad \eta\circ\ff=0.
\end{equation}

If $H=\ker(\eta)$ is the paracontact distribution of the tangent bundle of
$(M,\phi,\xi,\eta)$,
the endomorphism $\phi$ induces an almost product structure (in particular, an almost paracomplex structure) on each fiber of $H$ \cite{KaWi},
so that $(H,\phi)$ is a $2n$-dimensional almost product distribution (in particular, an almost paracomplex distribution).
Let us note that an almost paracomplex structure is an almost product structure $P$, \ie $P^2=\id$ and $P\neq\pm\id$,
such that the eigenvalues $+1$ and $-1$ of $P$ have  the same multiplicity $n$
\cite{CFG}, \ie $\tr P=0$ follows.

In the present work we consider the case of almost paracontact paracomplex manifolds, \ie its paracontact distribution is equipped with an almost paracomplex structure. According to S. Sasaki \cite{Sas80}, these manifolds are called almost paracontact manifolds of type $(n,n)$. 
For them we have $\tr{\ff}=0$.

Let $g$ be an associated Riemannian metric such that
\begin{equation*}\label{str2}
g(x,\xi)=\eta(x),\qquad g(\phi x,\phi y)=g(x,y)-\eta(x)\eta(y).
\end{equation*}
Then $(M,\phi,\xi,\eta,g)$ is called an \emph{almost paracontact paracomplex
Riemannian manifold} 
\cite{Sato76}. An almost paracontact paracomplex Riemannian manifold $(M,\ff,\xi,\eta)$  is called
\emph{paracontact paracomplex Riemannian manifold} 
if, in addition  the following  condition holds\cite{Sato77}
\begin{equation}\label{paracont}
2g(x,\phi y)=\left( \mathcal{L}_\xi g \right)(x,y)=\left( \nabla_x \eta \right)(y)+\left( \nabla_y \eta \right)(x),
\end{equation}
where $\mathcal L$ denotes the Lie derivative and $\nabla$ is the Levi-Civita connection of the Riemannian metric $g$.

The aim of the  paper is to define  a new class of paracontact paracomplex Riemannian manifolds which arise  under the condition that a certain Riemannian cone over it has a paraholomorphic paracomplex Riemannian (briefly, \emph{phpcR}) structure. We call it para-Sasaki-like Riemannian manifold and give explicit examples. Studying the structure of the para-Sasaki-like Riemannian spaces we show  that the paracontact form $\eta$ is closed and a para-Sasaki-like Riemannian manifold locally can be considered as a certain product of the real line with a phpcR manifold which locally is the  Riemannian product of two Riemannian spaces with equal dimension. We also get that the curvature of the para-Sasaki-like manifolds is completely determined by the curvature of the underlying local phpcR manifold as well as  the Ricci curvature in the direction of $\xi$ is equal to $-2n$ {while in the Sasaki case it is $2n$.}
In this sense, the para-Sasaki-like manifolds can be considered as the counterpart  of the Sasaki manifolds;
 the skew symmetric part of $\nabla \eta$ vanishes while in  the Sasaki case the symmetric (Killing) part vanishes

We define a hyperbolic extension of a (complete) phpcR manifold, which looks like as a certain warped product, showing that it is a (complete) para-Sasaki-like Riemannian manifold. Moreover, we  show that  if the starting phpcR manifold is a complete Einstein manifold with negative scalar curvature then its hyperbolic extension is a complete Einstein para-Sasaki-like Riemannian manifold with negative scalar curvature thus producing new examples of a complete Einstein Riemannian manifold with negative scalar curvature  (see Theorem~\ref{extein} and Example~3).

In the last section we define and study paracontact conformal/homothetic deformations extracting a subclass which preserve the para-Sasaki-like condition. In the case of paracontact homothetic deformation of a para-Sasaki-like Riemannian space  we obtain that the Ricci tensor is an invariant.


\begin{conv}
\label{conven} 
Let $\M$ be an apcpcR 
manifold.
\begin{enumerate}[a)]
\item 
We shall denote the smooth vector fields on
$M$ by $x$, $y$, $z$, $w$, \ie $x,y,z,w\in\X(M)$.

\item We shall use $X$, $Y$, $Z$, $W$
to denote smooth horizontal vector fields on $M$, \ie $%
X,Y,Z,W\in H=\ker(\eta)$.




\end{enumerate}
\end{conv}

\noindent{\bf Acknowledgments.} \vskip 0.1cm
\noindent The research of S.\,I.  is partially supported   by Contract DH/12/3/12.12.2017,
Contract 80-10-12/18.03.2020 with the Sofia University ``St. Kliment Ohridski''\
and
the National Science Fund of Bulgaria, National Scientific Program ``VIHREN'', Project No. KP-06-DV-7.
The research of H.\,M. is partially supported by the National Scientific Program ``Young Researchers and Post-Doctorants'' and
the project MU19-FMI-020 of the Scientific Research Fund, University of Plovdiv ``Paisii Hilendarski''.
The research of M.\,M. is partially supported by projects MU19-FMI-020 and
FP19-FMI-002 of the Scientific Research Fund, University of Plovdiv ``Paisii Hilendarski''.


\section{Almost paracontact  paracomplex Riemannian manifolds}

Let $(M,\ff,\xi,\eta)$ be a $(2n+1)$-dimension\-al almost paracontact paracomplex manifold, 
 \ie the eigenspaces of $\ff$ on the paracomplex distribution $H=\ker(\eta)$ have equal dimension $n$.

%

An almost paracontact paracomplex manifold 
is a \emph{normal almost paracontact paracomplex manifold} if the corresponding
almost paracomplex structure $\check P$  on $\check M=M\times
\R$
defined by
\begin{equation}\label{concom}
\check PX=\ff X,\qquad \check P\xi=r\ddr,\qquad \check P\ddr=\tfrac{1}{r}\xi
\end{equation}
is integrable (\ie
$(\check M,\check P)$  is a paracomplex manifold) \cite{CFG}.
The almost
paracontact paracomplex structure is normal if and only if the Nijenhuis
tensor $N$ of $(\ff,\xi,\eta)$  vanishes, where $N$ is defined by
\[
N = [\ff, \ff]-\D{\eta}\otimes\xi,\quad [\ff, \ff](x, y)=\left[\ff
x,\ff y\right]+\ff^2\left[x,y\right]-\ff\left[\ff
x,y\right]-\ff\left[x,\ff y\right],
\]
and  $[\ff,\ff]$ is the Nijenhuis torsion of $\ff$  \cite{Sato76}.

The associated metric $\widetilde{g}$ of $g$ on an almost paracontact paracomplex Riemannian manifold $\M$
 is defined by
$$\widetilde{g}(x,y)=g(x,\ff y)\allowbreak+\eta(x)\eta(y).$$ It is a
pseudo-Riemannian metric of signature $(n+1,n)$ (see e.g.  \cite{ManTav57}).

The almost paracontact  paracomplex Riemannian man\-i\-fold
 is known also as an almost
paracontact Riemannian mani\-fold of type $(n,n)$
\cite{ManSta01}. 
The structure group of these manifolds
is $O(n)\times O(n)\times I(1)$, where $O(n)$ and
$I(1)$ are the orthogonal matrix of size $n$ and the unit matrix
of size $1$, respectively.


The covariant derivatives of $\ff$, $\xi$, $\eta$ with respect to
the Levi-Civita connection $\n$ of $g$ play a fundamental role in the
differential geometry on the almost paracontact Riemannian manifolds.  The
structure tensor $F$ of type (0,3) on $\M$ is defined by
\begin{equation}\label{F=nfi}
F(x,y,z)=g\bigl( \left( \nabla_x \ff \right)y,z\bigr).
\end{equation}
It has the following properties \cite{ManSta01}:
\begin{equation}\label{F-prop}
F(x, y, z) = F(x, z, y) = -F(x,\ff y,\ff z)+\eta(y)F(x, \xi, z) +
\eta(z)F(x, y, \xi).
\end{equation}
The relations of $\n\xi$ and $\n\eta$ with $F$ are:
\begin{equation}\label{Fxieta}
    \left(\n_x\eta\right)(y)=g\left(\n_x\xi,y\right)=-F(x,\ff y,\xi).
\end{equation}

The  1-forms  associated with $F$:
$
\ta(z)=\sum_{i=1}^{2n}F(e_i,e_i,z)$, $
\ta^*(z)=\sum_{i=1}^{2n}F(e_i,\ff e_i,z)$, $
\om(z)=F(\xi,\xi,z)
$
satisfy the obvious relations $\ta^*\circ\ff=-\ta\circ\ff^2$ and $\om(\xi)=0$.


In \cite{ManTav57}, besides the Nijenhuis tensor $N$ of an almost paracontact Riemannian
structure, it is defined the symmetric (1,2)-tensor
$\widehat N$ as follows: consider the symmetric brackets $\{x,y
\}$ given by
\[
\begin{aligned}
g(\{x,y\},z)=g(\nabla_xy+\nabla_yx,z)=x\,g(y,z)+y\,g(x,z)-z\,g(x,y)+g([z,x],y)+g([z,y],x);
\end{aligned}
\]
set
\[
\{\ff ,\ff\}(x,y)=\{\ff x,\ff y\}+\ff^2\{x,y\}-\ff\{\ff
x,y\}-\ff\{x,\ff y\}
\]
and define the symmetric tensor $\widehat N$ as follows %
\[
\widehat N(x,y)=\{\ff,\ff\}(x,y)-\bigl(\left(\nabla_{x}\eta\right)(y)+\left(\nabla_{x}\eta\right)(y)\bigr)\xi=\{\ff,\ff\}(x,y)-(\mathcal L_{\xi}g)(x,y)\otimes \xi.
\]
%
The tensor $\widehat N$ is also called  \emph{the associated
Nijenhuis tensor} of the almost paracontact Riemannian structure.

\noindent We denote the corresponding tensors of type (0,3) by the same
letters, $N(x,y,z)=g(N(x,y),z)$,  $\widehat N(x,y,z)=g(\widehat
N(x,y),z)$. Both tensors $N$ and $\widehat N$ can be expressed in
terms of the fundamental tensor $F$ as follows 
\begin{gather}
N(x,y,z)=F(\ff x,y,z)-F(\ff y,x,z)-F(x,y,\ff z)+F(y,x,\ff z)
\label{enu}
+\eta(z)\bigl[F(x,\ff y,\xi)-F(y,\ff
x,\xi)\bigr],
\\[4pt]
\widehat N(x,y,z)=F(\ff x,y,z)+F(\ff y,x,z)-F(x,y,\ff z)-F(y,x,\ff z)
+\eta(z)\bigl[F(x,\ff y,\xi)+F(y,\ff x,\xi)\bigr].\label{enhat}
\end{gather}

\subsection{Relation with paraholomorphic paracomplex Riemannian manifolds}
Notice that the $2n$-dimensional  distribution
$H=\ker(\eta)$  is endowed with an almost paracom\-plex
structure
$P=\ff|_H$, a metric $h=g|_H$, where $\ff|_H$, $g|_H$ are  the
restrictions of $\ff$, $g$ on $H$, respectively. The metric
$h$ is compatible with $P$  as follows
\begin{equation}\label{hP}
h(PX,PY)=h(X,Y), \qquad \widetilde{h}(X,Y)=h(X,PY),
\end{equation}
where $\widetilde{h}$ is the associated neutral metric.

We recall that a $2n$-dimensional almost paracomplex manifold
$(N,P,h)$ endowed with a Riemannian metric $h$ 
satisfying \eqref{hP} is known as an almost paracomplex
Riemannian
manifold \cite{CFG,Lib52} or almost product  Riemannian manifold with $\tr\,{P}=0$ 
\cite{StaGri92,StaGriMek87,StaGriMek91}. When the almost product structure $P$ is
parallel with respect to the Levi-Civita connection $\nabla'$ of
the metric $h$, $\nabla'P=0$, then the manifold is known as a
Riemannian $P$-manifold \cite{StaGriMek87}, a locally product Riemannian manifold
or a paraholomorphic paracomplex Riemannian manifold
\cite{Oku67}. In this case the almost
product structure $P$ is integrable.

Let us denote the structure (0,3)-tensor of $(N,P,h)$ as follows
\begin{equation}\label{1.4}
F'(X,Y,Z)=h\bigl(\left(\n'_X P\right)Y,Z\bigr),
\end{equation}

The equalities  $P^2=\id$ and \eqref{1.4}  imply the
properties:
\begin{equation*}\label{1.6}
F'(X,Y,Z)=F'(X,Z,Y)=-F'(X,PY,PZ),\qquad
F'(X,PY,Z)=-F'(X,Y,PZ).
\end{equation*}
The  1-forms $\ta'$ and $\ta'^*$  are
given by
$    \ta'(Z)=\sum_{i=1}^{2n}F'(e_i,e_i,Z)$, $
    \ta'^*(Z)=\sum_{i=1}^{2n}F'(e_i,Pe_i,Z).
$
\subsection{The  case of parallel structures} The simplest case  of
almost paracontact Riemannian  manifolds is when the structures
are $\n$-parallel, $\n\ff=\n\xi=\n\eta=\n g=\n \widetilde{g}=0$,
and it is determined by the condition $F(x,y,z)=0$. In this case
the distribution $H$ is involutive. The corresponding integral
submanifold is a totally geodesic submanifold which inherits a
phpcR structure and the almost
paracontact Riemannian manifold is locally a Riemannian product of
a phpcR manifold with a real interval.

\section{Para-Sasaki-like 
Riemannian manifolds}

In this section we consider the Riemannian cone over an apcpcR manifold and determine a para-Sa\-sa\-ki-like
para\-contact paracomplex Riemannian manifold 
with the condition that its
Riemannian cone is a Riemannian manifold with a paraholomorphic paracomplex structure.

\subsection{Paraholomorphic  Riemannian cone}
Let $\M$ be a  $(2n+1)$-dimensional apcpcR 
manifold. We consider the Riemannian cone $\mathcal{C}(M)=M\times \R^+$ over $M$
equipped with the almost paracomplex structure $\check P$ determined in
\eqref{concom} and the Riemannian metric defined by
\begin{equation}\label{barg}
\begin{array}{l}
\check{g}\left(\left(x,a\ddr\right),\left(y,b\ddr\right)\right)
=r^2g(x,y)|_H+\eta(x)\eta(y)+ab
=r^2g(x,y)+(1-r^2)\eta(x)\eta(y)+ab,
\end{array}
\end{equation}
where $r$ is the coordinate on $\R^+$ and $a$, $b$ are
$C^{\infty}$ functions on $M\times \R^+$.


Using the general Koszul formula
\begin{equation}\label{koszul}
\begin{split}
2g(\nabla_xy,z)=xg(y,z)+yg(z,x)-zg(x,y)
+g([x,y],z)+g([z,x],y)+g([z,y],x),
\end{split}
\end{equation}
we calculate from \eqref{barg} that the  non-zero components of
the Levi-Civita connection $\nn$ of the Riemannian metric $\check
g$ on $\mathcal{C}(M)$ are given by

\[
\begin{array}{l}
    \g\left(\nn_X Y,Z\right)=r^2 g\left(\n_X Y,Z\right),\qquad
    \g\left(\nn_X Y,\ddr\right)=-r g\left(X, Y\right), %
    \\[4pt]
    \g\left(\nn_X Y,\xi\right)=r^2 g\left(\n_X
    Y,\xi\right)+\tfrac{1}{2}\left(r^2-1\right)\D\eta(X,Y),
\\[4pt]
    \g\left(\nn_X \xi,Z\right)=r^2 g\left(\n_X \xi,
    Z\right)-\tfrac{1}{2}\left(r^2-1\right)\D\eta(X,Z),
    \\[4pt]
    \g\left(\nn_{\xi} Y,Z\right)=
    r^2 g\left(\n_{\xi}Y,Z\right)-\tfrac{1}{2}(r^2-1)\D\eta(Y,Z),
    \\[4pt]
    \g\left(\nn_{\xi} Y,\xi\right)=-g\left(\n_{\xi}\xi,Y\right), \qquad
    \g\left(\nn_{\xi} \xi,Z\right)= g\left(\n_{\xi}\xi,Z\right),
    \\[4pt]
    \g\left(\nn_X \ddr,Z\right)=r g\left(X,Z\right), \qquad
    \g\left(\nn_{\ddr} Y,Z\right)=r g\left(Y,Z\right).
\end{array}
\]


Applying \eqref{concom}, we get 
that
the non-zero components of 
$\nn \check P$
are given by
\[
\begin{array}{rl}
    &\g\left(\left(\nn_X \check P\right)Y,Z\right)
    =r^2 g\left(\left(\n_X \ff\right)Y,Z\right),
    \\[4pt]
    &\g\left(\left(\nn_X  \check P\right)Y,\xi\right)
    =r^2 \left\{g\left(\left(\n_X \ff\right)Y,\xi\right)
    +g(X,Y)\right\}+\tfrac{1}{2}\left(r^2-1\right)\D\eta(X,\ff Y),%
    \\[4pt]
    &\g\left(\left(\nn_X  \check P\right)Y,\ddr\right)=r\left\{g\left(\n_X
    \xi,Y\right)
    -g\left(X,\ff Y\right)\right\}-\tfrac{1}{2r}(r^2-1)\D\eta(X,Y),%
\\[4pt]
    &\g\left(\left(\nn_X  \check P\right)\xi,Z\right)=-r^2\left\{g\left(\n_X \xi,\ff
    Z\right)
    -g\left(X,Z\right)\right\}+\tfrac{1}{2}(r^2-1)\D\eta(X,\ff Z),
    \\[4pt]
    &\g\left(\left(\nn_X  \check P\right)\ddr,Z\right)=r\left\{g\left(\n_X \xi,
    Z\right)
    -g\left(X,\ff Z\right)\right\}-\tfrac{1}{2r}(r^2-1)\D\eta(X,Z),
    \\[4pt]
    &\g\left(\left(\nn_{\xi}  \check P\right)Y,Z\right)
    =r^2g\left(\left(\n_{\xi}\ff\right)Y,Z\right)
    -\tfrac{1}{2}(r^2-1)\left\{\D\eta(\ff Y, Z)-\D\eta(Y,\ff Z)\right\},
    \\[4pt]
    &\g\left(\left(\nn_{\xi}  \check P\right)Y,\xi\right)=-g(\n_{\xi}\xi,\ff Y),
    \qquad
    \g\left(\left(\nn_{\xi}  \check P\right)\xi,Z\right)=-g\left(\n_{\xi}\xi, \ff
    Z\right),
    \\[4pt]
    &\g\left(\left(\nn_{\xi}  \check P\right)Y,\ddr\right)=
    \tfrac{1}{r}g\left(\n_{\xi}\xi, Y\right), \qquad  \g\left(\left(\nn_{\xi}
    \check P\right)\ddr,Z\right)=\tfrac{1}{r}g\left(\n_{\xi}\xi, Z\right).
\end{array}
\]
\begin{prop}\label{defs}
The  Riemannian cone $\mathcal{C}(M)$ over an apcpcR
manifold $\M$  is a Riemannian {manifold with a paraholomorphic paracomplex structure} if and only
if the following conditions hold
\begin{alignat}{1}
&F(X,Y,Z)=F(\xi,Y,Z)=\om(Z)=0\label{sasaki},\\[4pt]
&F(X,Y,\xi)=-g(X,Y)\label{sasaki0}.
\end{alignat}
\end{prop}

\begin{proof}[Proof]
The expressions above  yield that $\nn \check P=0$ on the Riemannian cone
$(\mathcal{C}(M),\check P,\check g)$  
if and only if  the apcpcR 
manifold $\M$ satisfies the following con\-ditions
\begin{alignat}{1}
    &F(X,Y,Z)=0, \qquad \om(Z)=0, \qquad \n_{\xi}\xi=0\label{nxixi}
   \\[4pt]
    &F(X,Y,\xi)=-g(X,Y)-\tfrac{1}{2r^2}\left(r^2-1\right)\D\eta(X,\ff
    Y)\label{FXYxi},\\[4pt]
    &F(\xi,Y,Z)=\tfrac{1}{2r^2}\left(r^2-1\right)\left\{\D\eta(\ff Y,Z)
    -\D\eta(Y,\ff Z)\right\}\label{FxiYZ}.
\end{alignat}

Further, according to \eqref{FXYxi}, we get
$
\left(\n_X\eta\right)(Y)=g(X,\ff Y)+\tfrac{1}{2r^2}\left(r^2-1\right)\D\eta(X,Y),
$
yielding $\D\eta(X,Y)=\tfrac{1}{r^2}\left(r^2-1\right)\D\eta(X,Y)$
since $\widetilde g$ is symmetric. The latter equality shows
$\D\eta(X,Y)=0$ 
yielding
\begin{equation}\label{sasaki1}
\left(\n_X\eta\right)(Y)=g(X,\ff  Y).
\end{equation}
{Therefore \eqref{paracont} holds and $\M$ is a paracontact Riemannian manifold.}

From \eqref{nxixi} we get
$\D\eta(\xi,X)=(\n_{\xi}\eta)(X)-(\n_X\eta)(\xi)=0$. Hence, we have
$\D\eta=0$. Substitute $\D\eta=0$ into \eqref{FXYxi}-\eqref{FxiYZ}
to complete the proof of the proposition.
\end{proof}

\begin{defn}
A manifold $\M$ is said to be \emph{para-Sasaki-like {paracontact para\-complex}
Riemannian manifold} (for short, \emph{para-Sasaki-like Riemannian manifold})
if the structure tensors $\ff, \xi, \eta,
g$ satisfy the equali\-ties \eqref{sasaki} and \eqref{sasaki0}.
\end{defn}

To characterize {para-Sasaki-like Riemannian manifolds} by the structure tensors, we need the
following general formula for any apcpcR manifold $\M$, known from \cite{ManTav57}
%
\begin{equation}\label{nabf}
\begin{split}
g(\n_x\ff)y,z)=F(x,y,z)&=\frac14\bigl[N(\ff x,y,z)+N(\ff x,z,y)
+\widehat N(\ff x,y,z)+\widehat N(\ff x,z,y)\bigr]\\[4pt]
&\phantom{=\ }
-\frac12\eta(x)\bigl[N(\xi,y,\ff z)+\widehat
N(\xi,y,\ff z)+\eta(z)\widehat N(\xi,\xi,\ff y)\bigr].
\end{split}
\end{equation}

The next result determines {the para-Sasaki-like Riemannian  manifolds} by the
structure tensors.

\begin{thm}\label{thm:ssss}
Let $\M$ be an apcpcR 
manifold. The following conditions are equivalent:
\begin{itemize}
\item[a)] The manifold $\M$ is para-Sasaki-like; %
\item[b)] The covariant derivative $\nabla \ff$ satisfies the
equality
\begin{equation}\label{defsl}
\begin{array}{l}
(\nabla_x\ff)y=-g(x,y)\xi-\eta(y)x+2\eta(x)\eta(y)\xi\\[4pt]
 \phantom{(\nabla_x\ff)y}
=-g(\ff x,\ff y)\xi-\eta(y)\ff^2 x;
\end{array}
\end{equation}
\item[c)] The Nijenhuis tensors $N$ and $\widehat N$ satisfy
    the
conditions:
\begin{equation}\label{sasnn}
N=0,\qquad
\widehat N={-4(\widetilde{g}-\eta\otimes\eta)\otimes\xi}.
\end{equation}
\end{itemize}
\end{thm}
\begin{proof}[Proof]
It is easy to check using \eqref{F-prop} that
\eqref{defsl} is equivalent to the system of the equations
\eqref{sasaki} and \eqref{sasaki0} which established the equivalence between
a) and b) in view of Proposition~\ref{defs}.

Substitute  \eqref{defsl} consequently into \eqref{enu} and
\eqref{enhat} to get \eqref{sasnn} which gives the
im\-pli\-cat\-ion
b) $\Rightarrow$ c).

Suppose \eqref{sasnn} holds. Then we get that \eqref{defsl}
follows from \eqref{sasnn} and \eqref{nabf}. This completes the proof.
\end{proof}
%

\begin{cor}
Let $\M$ be {para-Sasaki-like Riemannian manifold}. Then we have: 
\begin{itemize}
\item[a)] the manifold $\M$ is normal paracontact Riemannian manifold, $N=0$, $2\widetilde{g}|_H=\mathcal{L}_\xi g $,
            the fundamental 1-form $\eta$ is closed, $\D\eta=0$
						and the integral             curves of
            $\xi$ are geodesics, $\n_{\xi}\xi=0$; %
\item[b)]  the 1-forms $\theta$ and $\theta^*$ satisfy the
equalities $ \theta=-2n\,\eta$ and $\theta^*=0$.
\end{itemize}
\end{cor}




\subsection{Example 1: Solvable Lie group as a para-Sasaki-like Riemannian manifold}

Consider the solvable Lie group $G$ of
dimension $2n+1$
 with a basis of left-invariant vector fields $\{e_0,\dots, e_{2n}\}$
 defined by the commutators
\begin{equation}\label{com}
[e_0,e_1]=-e_{n+1},\; \dots,\; [e_0,e_n]=-e_{2n},\;
[e_0,e_{n+1}]=-e_1,\; \dots,\; [e_0,e_{2n}]=-e_n.
\end{equation}
%

Define an invariant
apcpcR 
structure
 on $G$  by
\begin{equation}\label{strEx1}
\begin{array}{rl}
&g(e_i,e_i)=1,
\quad
g(e_i,e_j)=0,\quad i,j\in\{0,1,\dots,2n\},\; i\neq j,
\\[4pt]
&\xi=e_0, \quad \ff  e_1=e_{n+1},\quad  \dots,\quad  \ff
e_n=e_{2n}.
\end{array}
\end{equation}

Using the Koszul formula \eqref{koszul}, we check that
\eqref{sasaki} and  \eqref{sasaki0} 
are fulfilled, \ie it is {para-Sasaki-like}.

Let $e^0=\eta$, $e^1$, $\dots$, $e^{2n}$ be the corresponding dual
1-forms, $e^i(e_j)=\delta^i_j$. From \eqref{com}
it follows that the structure equations of the group are
\begin{equation}\label{comstr}
\begin{array}{llll}
 \D e^0=\D\eta=0,\; &\D e^1=e^{0}\wedge e^{n+1},\;&\dots,\; &\D
e^n=e^{0}\wedge e^{2n},\\[4pt]
&\D e^{n+1}=e^{0}\wedge e^{1}, &\dots, &\D e^{2n}= e^{0}\wedge
e^{n}
\end{array}
\end{equation}
and the para-Sasaki-like  Riemannian structure
has the form
\begin{equation}\label{sas}
g=\sum_{i=0}^{2n}\left(e^i\right)^2,%
\qquad \ff e^0= 0,\ \ff e^1= e^{n+1},\ \dots,\ \ff e^n= e^{2n}.
\end{equation}


The basis of dual 1-forms can be the following
\begin{equation}\label{e^i_Ex1}
\begin{array}{ll}
e^0=\D t, 						\qquad &\quad e^i				=\cosh(t)\D x^i+\sinh(t)\D x^{n+i},\\[4pt]
i\in\{1,2,\dots,n\}, \qquad & e^{n+i}=\sinh(t)\D x^i+\cosh(t)\D x^{n+i}.
\end{array}
\end{equation}
%
%
The 1-forms defined in \eqref{e^i_Ex1} satisfy
\eqref{comstr} and the para-Sasaki-like  Riemannian
metric has the form
\begin{equation}\label{sasmetric}
g=\D t^2+\cosh(2t)\sum_{i=1}^{2n}\left(\D x^i\right)^2+ \sinh(2t)\sum_{i=1}^n\D x^i\D x^{n+i}.
\end{equation}

It follows from \eqref{com}, \eqref{sas}, \eqref{e^i_Ex1} and
\eqref{sasmetric} that the distribution
$H=
\Span\{e_1,\dots,\allowbreak{}e_{2n}\}$ is integrable and the
corresponding integral submanifold can be considered as the
{flat space} $\R^{2n}=\Span\{\D
x^1,\dots,\D x^{2n}\}$
with the  following phpcR
structure  %
\[
P\D x^1=\D x^{n+1},\; \dots,\; P\D x^n=\D x^{2n};\qquad
h=\sum_{i=1}^{2n}(\D x^i)^2, \qquad \widetilde h =2\sum_{i=1}^n\D
x^i\D x^{n+i}.
\]


\subsection{Hyperbolic extension of a paraholomorphic paracomplex Riemannian manifold}

Inspired by Example~1, we proposed the following more general
construction. Let $(N^{2n},J,h,\widetilde{h})$ be a
$2n$-dimen\-sion\-al phpcR
manifold,
\ie the almost pro\-duct structure $P$ has $\tr\, P=0$, acts as an
isometry on the metric $h$, $h(PX,PY)=h(X,Y)$  and it is parallel
with respect to the Levi-Civita connection of $h$. In particular,
the almost paracomplex structure $P$ is integrable. The
associated
neutral pseudo-Riemannian metric $\widetilde{h}$ is defined by
$\widetilde{h}(X,Y)=h(PX,Y)$ and it is also parallel with respect
to the Levi-Civita connection of $h$.

Consider the product manifold $M^{2n+1}=\mathbb R\times N^{2n}$.
Let $\D t$ be the coordinate 1-form on $\mathbb R$ and define an
apcpcR 
structure on
$M^{2n+1}$ as follows
\begin{equation}\label{strsas}
\eta=\D t, \quad \ff |_H
=P, \quad \eta\circ\ff =0,\quad g=\D t^2+\cosh(2t)\,h+\sinh(2t)\,\widetilde{h}.
\end{equation}

\begin{thm}\label{ext}
Let $(N^{2n},P,h,\widetilde{h})$ be a $2n$-dimensional phpcR manifold. Then the product
manifold $M^{2n+1}=\mathbb R\times N^{2n}$ equipped with the
apcpc 
 Riemannian structure defined in
\eqref{strsas} is a para-Sasaki-like  Riemannian
manifold.
If the Riemannian manifold $(N^{2n},h)$ is complete then the para-Sasaki-like Riemannian manifold  $(M^{2n+1},g)=(\mathbb R\times N^{2n},g)$ is complete.
\end{thm}
\begin{proof}[Proof]
To show that the metric $g$ is Riemannian we consider an orthonormal basis for $h$ of the form $\{e_1,Pe_1,\dots,e_n,Pe_n\}$. Then the matrix of $g$ with respect to the basis $\{\xi=\partial_t,e_1,Pe_1,\dots,e_n,Pe_n\}$ has the form
\[
\left(
\begin{array}{c|cccc}
1&o&o&\cdots &o\\\hline\\[-9pt]
o^\top&A&O&\cdots &O\\\\[-9pt]
o^\top&O&A&\cdots &O\\
\vdots&\vdots&\vdots&\ddots&\vdots\\\\[-9pt]
o^\top&O&O&\cdots &A
\end{array}
\right),
\]
where  we have denoted
\[
A=\left(
\begin{array}{cc}
\cosh(2t)&\sinh(2t)\\
\sinh(2t)&\cosh(2t)
\end{array}
\right),
\quad
O=\left(
\begin{array}{cc}
0&0\\
0&0
\end{array}
\right),
\quad
o=
\left(
\begin{array}{cc}
0&0
\end{array}
\right),
\quad
o^\top=
\left(
\begin{array}{c}
0\\
0
\end{array}
\right).
\]
The matrix of $g$ is clearly positive definite due to the identity $\cosh^2(2t)-\sinh^2(2t)=1$ implying that all its principal minors are positive.

It is easy to check using \eqref{koszul}, \eqref{strsas} and the
fact that the paracomplex structure $P$ is parallel with
respect to
the Levi-Civita connection of $h$ that the structure defined in
\eqref{strsas} satisfies \eqref{sasaki} and \eqref{sasaki0} and
thus $\M$  is a {para-Sasaki-like Riemannian manifold}.

To show that the metric $g$ on $M^{2n+1}=\mathbb R\times N^{2n}$ is complete we  observe  the metric $\D t^2$ on $\mathbb R$ is complete and  if the Riemannian metric $h$ on $N^{2n}$ is complete 
then the Riemannian metrics on $N^{2n}$ from the one-parameter family 
\[
g|_{_N}(t)=\cosh(2t)\,h+\sinh(2t)\,\widetilde{h}
\] 
 are complete since their Levi-Civita connections coincide with the Levi-Civita connection of $h$ (cf. \eqref{ein1} below) and then apply \cite[Lemma~2]{CHM}.
\end{proof}
We call the para-Sasaki-like  Riemannian manifold con\-struct\-ed in
Theorem~\ref{ext} by a phpcR manifold
\emph{{a hyper\-bol\-ic extension} of a paraholomorphic
paracomplex Riemannian
manifold}.


%
%
%

\subsection{Example 2: Lie group of dimension 5 as a hyperbolic extension of a phpcR manifold}

Let us consider the Lie group $G^5$ of
dimension $5$
 with a basis of left-invariant vector fields $\{e_0,\dots, e_{4}\}$
 defined by the commutators
\begin{equation}\label{comEx2}
\begin{array}{ll}
[e_0,e_1] = \lm e_2 - e_3 + \mu e_4,\quad &
[e_0,e_2] = - \lm e_1 - \mu e_3 - e_4,\\[4pt]
[e_0,e_3] = - e_1  + \mu e_2 + \lm e_4,\quad &
[e_0,e_4] = -\mu e_1 - e_2 - \lm e_3,
\qquad \lm,\, \mu\in\R.
\end{array}
\end{equation}
We equip $G^5$ with an invariant apcpcR 
structure as in \eqref{strEx1} for $n=2$. Then, using \eqref{koszul}, we calculate
that the non-zero components of the Levi-Civita connection are
\[
\begin{array}{c}
\begin{array}{llll}
\n_{e_0} e_1 = \lm e_2+\mu e_4,\quad & \n_{e_1}e_0 = e_3,\quad &
\n_{e_0} e_2 = -\lm e_1-\mu e_3, \quad & \n_{e_2} e_0 = e_4,\\[4pt]
\n_{e_0} e_3 = \mu e_2 + \lm e_4,\quad & \n_{e_3} e_0 = e_1, \quad &
\n_{e_0} e_4 = -\mu e_1 - \lm e_3, \quad & \n_{e_4}e_0 = e_2,
\end{array}
\\[4pt]
\begin{array}{c}\\[-8pt]
\n_{e_1}e_3 = \n_{e_2} e_4 = \n_{e_3} e_1 = \n_{e_4}e_2 = - e_0.
\end{array}
\end{array}
\]
Similarly as in Example~1 we verify that the constructed manifold
$(G^5,\ff,\xi,\eta,g)$ is a para-Sasaki-like 
Rie\-mannian manifold.

We consider the case for $\mu=0$ and $\lambda\not=0$.
By virtue of \eqref{comEx2}, the structure equations
of the group become
\begin{equation}\label{comstr2}
\begin{array}{ll}
\D e^0=\D\eta=0,\; & \\[4pt]
\D e^1=\lm\, e^{0}\wedge  e^2 + e^{0}\wedge e^3, & \qquad
\D e^2= -\lm\, e^{0}\wedge e^1 + e^{0}\wedge e^4,
\\[4pt]
\D e^3= e^{0}\wedge e^{1} + \lm\, e^{0}\wedge e^4, &\qquad
\D e^4=e^{0}\wedge e^2 - \lm\, e^{0}\wedge e^3.
\end{array}
\end{equation}
%
%
%
A basis of 1-forms satisfying \eqref{comstr2} is given by
$e^0=\D t$ and
\[
\begin{split}
e^1=\ &f_1\ \D x^1+f_2\ \D x^2 +f_3\ \D x^3+f_4\ \D x^4, \qquad
e^2=-f_3\ \D x^1-f_4\ \D x^2 +f_1\ \D x^3+f_2\ \D x^4,\\
e^3=\ &f_1\ \D x^1-f_2\ \D x^2 +f_3\ \D x^3-f_4\ \D x^4,\qquad 
e^4=-f_3\ \D x^1+f_4\ \D x^2 +f_1\ \D x^3-f_2\ \D x^4,
\end{split}
\]
where
\[
\begin{array}{ll}
f_1=\exp(t)\cos(\lm t),\quad  f_2=\exp(-t)\cos(\lm t), \quad
f_3=\exp(t)\sin(\lm t), \quad  f_4=\exp(-t)\sin(\lm t).
\end{array}
\]
Then the para-Sasaki-like Riemannian  metric is of the form
\begin{equation*}\label{m02}
g=\D t^2
+2\exp(2t)\left(\D x^1\right)^2
+2\exp(-2t)\left(\D x^2\right)^2
+2\exp(2t)\left(\D x^3\right)^2
+2\exp(-2t)\left(\D x^4\right)^2,
\end{equation*}
which can be written as follows
\begin{equation}\label{m2}
\begin{split}
g=
\D t^2%
&+2\cosh(2t)\left\{\left(\D x^1\right)^2+\left(\D x^2\right)^2+\left(\D x^3\right)^2+\left(\D x^4\right)^2\right\}\\
&+2\sinh(2t)\left\{\left(\D x^1\right)^2-\left(\D x^2\right)^2+\left(\D x^3\right)^2-\left(\D x^4\right)^2\right\}.
\end{split}
\end{equation}
It is clear from  \eqref{comstr2} that the distribution
$H=
\Span\{e_1,\dots,e_4\}$ is integrable and the corresponding
integral submanifold can be considered as the phpcR flat space $\R^{4}=\Span\{\D x^1,\allowbreak{}\dots,\allowbreak{}\D x^{4}\}$ with
the phpcR structure  given by
\begin{equation*}
\begin{split}
&
P\D x^1=\D x^1, \quad P\D x^2=-\D x^2,\quad P\D x^3=\D x^3, \quad P\D x^4=-\D x^4;\\
&
h=\left(\D x^1\right)^2+\left(\D x^2\right)^2+\left(\D x^3\right)^2+\left(\D x^4\right)^2.
\end{split}
\end{equation*}
Therefore, the associated metric $\widetilde{h}(X,Y)=h(X,PY)$ is
$
\widetilde{h}=\left(\D x^1\right)^2-\left(\D x^2\right)^2+\left(\D x^3\right)^2-\left(\D x^4\right)^2.$
Then,
the para-Sasaki-like Riemannian  metric \eqref{m2} takes the form as in \eqref{strsas}.

\section{Curvature properties of para-Sasaki-like Riemannian manifolds. Einstein condition}

Here we consider an apcpcR mani\-fold $\M$ of dimension $2n+1$. Its
curva\-ture tensor of type $(1,3)$ is defined as usual by
$R=[\nabla,\nabla]-\nabla_{[\ ,\ ]}$.  The corresponding curvature
tensor of type $(0,4)$ is denoted by the same letter and it is determined by
$R(x,y,z,w)=g(R(x,y)z,w)$. The Ricci tensor $Ric$,  the scalar
curvature $Scal$ and the *-scalar curvature $Scal^*$ are the usual
traces of the curvature
\[
Ric(x,y)=\sum_{i=0}^{2n}R(e_i,x,y,e_i),\quad
Scal=\sum_{i=0}^{2n}Ric(e_i,e_i),\quad
Scal^*=\sum_{i=0}^{2n} Ric(e_i,\ff e_i)
\]
with respect to an arbitrary orthonormal basis $\{e_0,\dots, e_{2n}\}$ of its tangent space.

\begin{prop}\label{vercurv}
On a para-Sasaki-like Riemannian  manifold $\M$
the following for\-mu\-la holds
\begin{equation}\label{curf}
\begin{array}{l}
R(x,y,\ff z,w)-R(x,y,z,\ff w)
=-\left\{g(y,z)-2\eta(y)\eta(z)\right\}g(x,\ff w)
-\left\{g(y,w)-2\eta(y)\eta(w)\right\}g(x,\ff z)\\[4pt]
\phantom{R(x,y,\ff z,w)-R(x,y,z,\ff w)=}
+\left\{g(x,z)-2\eta(x)\eta(z)\right\}g(y,\ff w)+\left\{g(x,w)-2\eta(x)\eta(w)\right\}g(y,\ff z).
\end{array}
\end{equation}
In particular, we have
\begin{eqnarray}
 &R(x,y)\xi=-\eta(y)x+\eta(x)y,  \label{cur} \\[4pt]
 &[X,\xi]\in H, \quad \nabla_{\xi}X=\ff X-[X,\xi] \in H, \label{inH}  \\[4pt]
&R(\xi,X)\xi=X, \quad Ric(y,\xi)=-2n\,\eta(y),\quad
Ric(\xi,\xi)=-2n.\label{ricxi}
\end{eqnarray}
\end{prop}

\begin{proof}[Proof]
Applying \eqref{defsl} to the Ricci identity for $\ff$, \ie
\[
R(x,y,\ff z,w)-R(x,y,z,\ff
w)=g\Bigl(\left(\n_x\n_y\ff\right)z,w\Bigr)-g\Bigl(\left(\n_y\n_x\ff\right)z,w\Bigr),
\]
and using
\eqref{sasaki1}, we obtain \eqref{curf} by  straightforward
calculations. Equality \eqref{curf} for $z=\xi$
implies \eqref{cur} due to \eqref{str}. The assertions in \eqref{inH} follow from
\eqref{sasaki1} and  $\D \eta=0$. Equalities
\eqref{ricxi} are direct consequences of
\eqref{cur}.
\end{proof}

\subsection{The horizontal curvature and the Einstein condition}

From $\D\eta=0$ it follows locally $\eta=\D t$, where $t$ is the coordinate of $\R$.
Then, $H=\ker\eta$ is integrable
and we get locally the product
$M^{2n+1}=\mathbb R \times N^{2n}$ with $TN^{2n}=H$.
As a result,	the submanifold
$(N^{2n},P=\ff|_H,h=g|_H)$ is a phpcR manifold.
In fact, by \eqref{sasaki} we get that
$h\left((\nabla^h_X P)Y,Z\right)=F(X,Y,Z)=0$, where
$\nabla^h$ is the Levi-Civita connection of $h$.

The submanifold $N^{2n}$ can be considered as a hypersurface of $M^{2n+1}$ with unit normal
$\xi=\frac{\D}{\D t}$. The equality \eqref{sasaki1}
yields
\[
g(\n_X\xi,Y)=-g(\n_XY,\xi)=g(X,\ff Y)=\widetilde{g}|_{H}(X,Y),
\qquad \n_{\xi}\xi=0.
\]
Therefore, the second fundamental form is equal to
$-\widetilde{g}|_{H}=-\widetilde h$.
Then, the Gauss equation (see e.g.
\cite[Chapter VII, Proposition~4.1]{KN}) has the form
\begin{equation}\label{gaus}
\begin{array}{l}
R(X,Y,Z,W)=R^h(X,Y,Z,W)+g(X,\ff Z)g(Y,\ff W)
-g(Y,\ff Z)g(X,\ff W),
\end{array}
\end{equation}
where $R^h$ is the curvature tensor of the phpcR manifold $(N^{2n},P,h)$.

For the horizontal Ricci tensor we obtain from  \eqref{ricxi} and \eqref{gaus} that
\begin{equation}\label{ric}
\begin{split}
Ric(Y,Z)&=\sum_{i=1}^{2n}R(e_i,Y,Z,e_i)+R(\xi,Y,Z,\xi)\\[4pt]
\phantom{Ric(Y,Z)}
&=Ric^h(Y,Z)+g(\ff Y,\ff Z)-g(Y,Z)=Ric^h(Y,Z),
\end{split}
\end{equation}
where $Ric^h$ is the Ricci tensor of $h=g|_{H}$.

Bearing in mind \propref{vercurv}, we find that the curvature tensor in
the direction of $\xi$ on a para-Sasaki-like Riemannian  manifold is completely determined by
$\eta,\ff,g,\widetilde g$.
Indeed, we obtain the following equality due to \eqref{cur}
and the properties of the Riemannian curvature
\begin{equation}\label{Rxi}
R(x,y,z,\xi)=R(\xi,z,y,x)=-\eta(x)g(y,z)+\eta(y)g(x,z).
\end{equation}
The formulas in  \eqref{gaus} and \eqref{Rxi} imply that the Riemannian curvature of
a para-Sasaki-like Riemannian  manifold is
completely determined by the curvature of the underlying
phpcR manifold $(N^{2n}, TN^{2n}=H,P,h)$ as follows 
\[
\begin{split}
R(x,y,z,w)&=R^h(x|_H, y|_H, z|_H, w|_H)
- g(y,\ff z)g(x,\ff w)+g(x,\ff z)g(y,\ff w)\\[4pt]
&\phantom{=}
-\{g(y,z)\eta(x)-g(x,z)\eta(y)\}\eta(w)
-\{g(x,w)\eta(y)-g(y,w)\eta(x)\}\eta(z).
\end{split}
\]
Then, for the Ricci tensor and the scalar curvatures we have
\begin{equation}\label{ricPinv}
\begin{split}
Ric(y,z)=Ric^h( y, z)-2n\,\eta(y)\eta(z),\qquad
Scal=Scal^h-2n,\qquad Scal^*=Scal^{h*}.
\end{split}
\end{equation}
We get from \eqref{ricPinv}, or comparing \eqref{ric} with \eqref{ricxi}, the following
\begin{prop}\label{einstein}
A   para-Sasaki-like Riemannian manifold $\M$ is Einstein if and only if the un\-der\-lying local phpcR manifold  $(N^{2n},P,h)$ is Einstein with negative scalar curvature $-4n^2$, \ie
\begin{equation}\label{ein0}
Ric^h =-2n\, h.
\end{equation}
\end{prop}
Proposition~\ref{einstein} allows a construction of a new Einstein manifold (see Example~3 below). We have
\begin{thm}\label{extein}
Let  $(N^{2n},P,h^N)$ be a $2n$-dimensional  Einstein  phpcR manifold with negative scalar curvature $-4n^2$, \ie its Ricci tensor satisfies \eqref{ein0}. Then its hyperbolic exten\-sion, the $(2n+1)$-dimensional space $(M^{2n+1}=\mathbb R\times N^{2n},g,\phi,\eta)$
with the apcpcR 
structure $(g,\phi,\eta)$ on
$M^{2n+1}$ defined by
\begin{equation*}\label{strsas1}
\eta=\D t, \quad \ff |_H
=P, \quad \eta\circ\ff =0,\quad g=\D t^2+\cosh(2t)\,h^N+\sinh(2t)\,\widetilde{h}^N
\end{equation*}
is an Einstein  para-Sasaki-like Riemannian manifold with negative scalar curvature.

If the Einstein  Riemannian manifold $(N^{2n},h)$ is complete then the para-Sasaki-like Riemannian manifold  $(M^{2n+1},g)=(\mathbb R\times N^{2n},g)$ is a complete Einstein Riemannian manifold with negative scalar curvature.
\end{thm}
\begin{proof}[Proof.]
According to Theorem~\ref{ext}, it remains to show that the Einstein condition on the Riemannian manifold $(M^{2n+1},g)$ holds.

The horizontal metrics, \ie the Riemannian metric $h$ and the pseudo-Riemannian metric $\widetilde{h}$ of signature $(n,n)$ on $N^{2n}$ are
\begin{equation}\label{einm}
\begin{split}
h=g{|_H}=\cosh(2t)h^N+\sinh(2t)\widetilde{h}^N,\qquad
\widetilde{h}=\widetilde{g}{|_H}=\sinh(2t)h^N+\cosh(2t)\widetilde{h}^N.
\end{split}
\end{equation}
The Levi-Civita connection $\nabla^{h^N}$ of the metric
$h^N$ coincides with the Levi-Civita connection of
$\widetilde{h}^N$ since $\nabla^{h^N} P=0$. Using this
fact,  the Koszul formula gives for $X,Y,Z\in TN^{2n}$  the following
\begin{equation}\label{ein1}
\begin{array}{l}
2g\left(\nabla^g_XY,Z\right)
=\cosh(2t)\,h^N\left(\nabla^{h^N}_XY,Z\right)+\sinh(2t)\,\widetilde{h}^N\left(\nabla^{h^N}_XY,Z\right)
=2h\left(\nabla^{h^N}_XY,Z\right),\\[4pt]
2g\left(\nabla^g_X\xi,Y\right)=\xi\, g(X,Y)=2\sinh(2t)\,h^N(X,Y)+2\cosh(2t)\,\widetilde{h}^N(X,Y)=2\widetilde{g}(X,Y).
\end{array}
\end{equation}
The first equality in \eqref{ein1} shows that the Levi-Civita connection $\nabla^h$ of the horizontal metric $h$ coincides with the Levi-Civita connection $\nabla^{h^N}$, $\nabla^h=\nabla^{h^N}$. 
Now, \eqref{einm} yields the following formula for the curvature of $h$
\begin{equation}\label{ein2}
R^h=\cosh(2t)\,R^{h^N}+\sinh(2t)\,\widetilde{R}^{h^N}, \qquad \widetilde{R}:=PR.
\end{equation}
For the Ricci tensor we get the following taking the trace in \eqref{ein2}
\begin{equation}\label{ein3}
Ric^h(X,Y)=\cosh(2t)\,Ric^{h^N}(X,Y)+\sinh(2t)\,Ric^{h^N}(X,PY).
\end{equation}
Now, \eqref{ein0}, \eqref{einm} and \eqref{ein3} imply
\begin{equation}\label{ein4}
Ric^h(X,Y)=-2n\left\{\cosh(2t)h^N(X,Y)+\sinh(2t)\widetilde{h}^N(X,Y)\right\}=-2n\,h(X,Y).
\end{equation}
The second equality in \eqref{ein1} tells us that the manifold $N^{2n}$ can be considered as a hypersurface of $M^{2n+1}$ with second fundamental form egual to $-\widetilde{g}$, which combined with \eqref{ein4} and Proposition~\ref{einstein} yields that the para-Sasaki-like Riemannian manifold $(M^{2n+1},\phi,\xi=\frac{\D}{\D t},\eta=\D t,g)$ is an Einstein Riemannian manifold with negative scalar curvature $-2n(2n+1)$.
\end{proof}

\subsection{Example~3: Complete para-Sasaki-like Einstein space as a hyperbolic extension}
Consider the product of two complete 
$n$-dimensional  Einstein Riemannian manifolds with a negative scalar curvature equal to $-2n^2$. For example, taking the product of two discs with the Poincare metric, $N^{2n}=D^n\times D^n$, $h^N=g_D\times g_D$ and the usual product structure $P$, defined by $PA=A$, $PB=-B$ for $(A,B) \in TD^n\times TD^n$, one gets a complete Einstein phpcR manifold $(N^{2n},P,h^N)$, whose Ricci tensor satisfies \eqref{ein0}.
The product manifold $M^{2n+1}=\mathbb R\times N^{2n}$ with the metric
$g=\D t^2+\cosh(2t)\,h^N+\sinh(2t)\,\widetilde{h}^N$ is a complete Einstein para-Sasaki-like  Riemannian manifold according to Theorem~\ref{extein}.


\subsection{Example~4: Hyperbolic extension of a $P$-invariant sphere in a flat space}

The present exam\-ple
illustrates Theorem~\ref{ext}.
Let us consider the real space  $\R^{2n+2}=\left\{\left(x^1,\dots,x^{2n+2}\right)\right\}$, $n\geq 2$, as a
flat phpcR manifold. It means that $\R^{2n+2}$ is
equipped with the canonical paracomplex structure $P'$ and the
canonical  $P'$-com\-patible Rie\-man\-ni\-an metrics $h'$ and $\wh'$  defined
for arbitrary vectors $x' = (x^1, \dots, x^{2n+2})$ and $y' = (y^1,
\dots, y^{2n+2})$ 
in $\R^{2n+2}$
as follows
\begin{align}
&P'x'=\left(x^{n+2},\dots,x^{2n+2},x^1,\dots,x^{n+1}\right),
\nonumber\\
&h'(x',y')=\sum_{i=1}^{2n+2} \left(x^i y^i\right), \qquad 
\widetilde h'(x',y')=\sum_{i=1}^{n+1}\left(x^i y^{n+i+1}+x^{n+i+1}y^i\right).\nonumber
\end{align}
Clearly, $P'$, $h'$, $\wh'$ satisfy \eqref{hP}, 
the Levi-Civita connection $\nabla'$ of the Riemannian metric $h'$ preserves the para\-complex structure $P'$, $\nabla' P'=0$ and we have a phpcR manifold.

The so-called \emph{invariant hypersurface} $S_h^{2n}(z'_0; a,b)$ in the phpcR manifold $(\R^{2n+2},h',\allowbreak{}P')$ is studied in  \cite{StaGriMek87,StaGriMek91}. We outline the construction below as follows.

Identifying the point $z' = (z^1, \dots,
z^{2n+2})$ in $\R^{2n+2}$ with its position vector $z'$, we
consider the  $P'$-invariant hypersurface $S_h^{2n}(z'_0;\A a, b)$
defined by the equations
\[
h'\left(z'-z'_0, z'-z'_0\right) = a,\qquad \widetilde h' \left(z'-z'_0, z'-z'_0\right) = b,
\]
where $(0, 0)\allowbreak\neq\allowbreak (a, b) \in\R^2$, $a>|b|$.
The codimension two submanifold $S_h^{2n}(z'_0; a,b)$ is the intersection of the standard $(2n+1)$-dimensional sphere with the standard hyperboloid   in $\R^{2n+2}$ and it is clearly
$P'$-invariant. The restriction of $h'$ on $S_h^{2n}(z'_0;a,b)$ has rank $2n$ due to the condition
$(0, 0)\allowbreak\neq\allowbreak (a, b)$.
The phpcR structure $(P',h')$ on $\R^{2n+2}$ inherits a phpcR  structure
$\bigl(P=P'|_{{S_h^{2n}}}, h=h'|_{S_h^{2n}}\bigr)$ on 
$S_h^{2n}(z'_0;a,b)$ for $n\geq 2$
 which, sometimes, is
called a \emph{$P$-invariant sphere} with center $z'_0$ and pair of
parameters $(a, b)$  \cite{StaGriMek87}.

The curvature tensor of $S_h^{2n}(z'_0;a,b)$ is given by the formula \cite{Sta87} (see also \cite{StaGriMek91})
\begin{equation}\label{Rnu}
R'|_{S_h^{2n}} = \frac{1}{a^2-b^2} \left\{a\left(\pi_1^{h'} + \pi_2^{h'}\right) - b \pi_3^{h'}\right\},
\end{equation}
where $2\pi_1^{h'} = h'|_{S_h^{2n}}\owedge h'|_{S_h^{2n}}$,
$2\pi_2^{h'} = \widetilde{h}'|_{S_h^{2n}}\owedge \widetilde{h}'|_{S_h^{2n}}$,
$\pi_3^{h'} =h'|_{S_h^{2n}}\owedge\widetilde{h}'|_{S_h^{2n}}$ and $\owedge$
stands for the Kulkarni-Nomizu product of two $(0,2)$-tensors; for
example,
\[
\begin{array}{l}
\left(h\owedge\widetilde{h}\right)(X,Y,Z,W)
=h(Y,Z) \widetilde h(X,W) - h'(X,Z)\widetilde h(Y,W)
+\widetilde h(Y,Z) h(X,W) - \widetilde h(X,Z) h(Y,W).
\end{array}
\]
Consequently, we have
\begin{equation}\label{Ric-h-sph}
Ric'|_{S_h^{2n}}=\frac{2(n-1)}{a^2-b^2}\left(a\, h'|_{S_h^{2n}}-b\, \widetilde{h}'|_{S_h^{2n}}\right),
\quad
Scal'|_{S_h^{2n}}=\frac{4n(n-1)a}{a^2-b^2},
\quad
Scal'^*|_{S_h^{2n}}=-\frac{4n(n-1)b}{a^2-b^2}.
\end{equation}

The  product manifold $M^{2n+1}=\mathbb R\times S_h^{2n}(z'_0;
a,b)$  equipped with the  apcpcR structure $(\ff,\xi,\eta,g)$ given in \eqref{strsas}
is a para-Sasaki-like 
Riemannian manifold according to \thmref{ext}.


Following the proof of \thmref{extein}, we get from \eqref{ein2} and \eqref{Rnu} the next formula for the horizontal curvature
\begin{equation}\label{rrr}
\begin{split}
R^h&=\frac{1}{a^2-b^2}\left\{\cosh(2t)\,\left[a\left(\pi_1^{h'}+\pi_2^{h'}\right)-b\,\pi_3^{h'}\right]
+\sinh(2t)\,\left[a\,\pi_3^{h'}-b\left(\pi_1^{h'}+\pi_2^{h'}\right)\right]\right\}\\[4pt]
\phantom{=}
&=\frac{1}{a^2-b^2}\left\{\left[a\cosh(2t)-b\sinh(2t)\right]\left(\pi_1^{h'}+\pi_2^{h'}\right)
-\left[b\cosh(2t)-	a\sinh(2t)\right]\pi_3^{h'}\right\}.
\end{split}
\end{equation}
Taking into account \eqref{gaus}, \eqref{einm} and \eqref{rrr},
we obtain the  expression of
the horizontal curvature $R|_H$ of the para-Sasaki-like Riemannian  manifold $M^{2n+1}=\mathbb R^+\times
S_h^{2n}(z'_0; a,b)$ 	
\[
\begin{split}
R|_H&=R^h+\sinh^2(2t)\,\pi_1^h+\cosh^2(2t)\,\pi_2^h-\sinh(2t)\cosh(2t)\,\pi_3^h\\[4pt]
\phantom{R|_H}
&=\frac1{a^2-b^2}\Bigl\{\bigl[a \cosh(2t)+b \sinh(2t)\bigr]\left(\pi_1^h+\pi_2^h\right)
-\bigl[b \cosh(2t)+a \sinh(2t)\bigr]\pi_3^h\Bigr\}.
\end{split}
\]
Then, \eqref{ric}, \eqref{einm} and \eqref{Ric-h-sph}  imply the following formula
for the horizontal Ricci tensor
\[
\begin{split}
Ric|_H&=Ric^h
=\frac{2(n-1)}{a^2-b^2}\Bigl\{\bigl[a\cosh(2t)+b\sinh(2t)\bigr]h
-\bigl[b\cosh(2t)+a\sinh(2t)\bigr]\widetilde h\Bigr\}.
\end{split}
\]
Thus, the latter equality, \eqref{ricPinv}, \eqref{einm}  and \eqref{rrr} give
\[
\begin{split}
Ric&=\frac{2(n-1)}{a^2-b^2}\Bigl\{\bigl[a\cosh(2t)+b\sinh(2t)\bigr](g-\eta\otimes\eta)
-\bigl[b\cosh(2t)+a\sinh(2t)\bigr]\widetilde g\Bigr\} -2n\, \eta\otimes\eta.
\end{split}
\]
Therefore, the para-Sasaki-like Riemannian manifold $M^{2n+1}=\mathbb R\times
S_h^{2n}(z'_0; a,b)$ 	is almost Ein\-stein-like since
its Ricci tensor has is expressed by the following way
$Ric=\al(t) g + \bt(t) \wg +\gm(t) \eta\otimes\eta$, where $\al(t)$, $\bt(t)$ and $\gm(t)$ are the smooth functions
determined in the above equality.

\section{Paracontact conformal  transformations}
Let $\M$ be an apcpcR manifold.
The transformation 
\begin{equation}\label{cct}
\begin{array}{l}
\overline{\eta}=\exp(w)\eta,\quad \overline{\xi}=\exp(-w)\xi,\\[4pt]
\overline{g}(x,y) = \exp(2u) \cosh(2v)g(x,y)+\exp(2u) \sinh(2v)g(x,\ff y)\\[4pt]
\phantom{\overline{g}(x,y) = \exp(2u) \cosh(2v)g(x,y)}
+\bigl\{\exp(2w) -\exp(2u) \cosh(2v)\bigr\}\eta(x)\eta(y),
\end{array}
\end{equation}
where $u$, $v$, $w$ are smooth on $M$
we call a \emph{paracontact conformal transformation} of  $(\ff, \xi, \eta, g)$. It is easy to check that $(M, \ff, \overline{\xi}, \overline{\eta}, \overline{g})$ is again an apcpcR manifold and the paracontact conformal transformations on an apcpcR manifold form a group.
When $u$, $v$, $w$ are constant we have a
\emph{paracontact homothetic transformation}.

In this section we study  the para-Sasaki-like condition under paracontact conformal transformations.

\begin{lem}\label{lemma-F}
Let $\M$ and $(M, \ff, \overline{\xi}, \overline{\eta}, \overline{g})$ be related by a paracontact con\-for\-mal transformation. Then we have

\begin{equation}\label{ff}
\begin{array}{l}
    2\overline{F}(x,y,z)=\exp(2u)\bigl\{\cosh(2v)\left[2F(x,y,z)-F_2(x,y,z)\right]+\sinh(2v) F_1(x,y,z)\\[4pt]
   \phantom{2\overline{F}(x,y,z)=\exp(2u)\bigl\{}
    +2\left[
    \chi_1(z)g(\ff x,\ff y)+\chi_1(y)g(\ff x,\ff z)+\chi_2(z)g(x,\ff y)+\chi_2(y)g(x,\ff z)\right]\bigr\}\\[4pt]
   \phantom{2\overline{F}(x,y,z)=}
    +\exp(2w)\bigl\{F_2(x,y,z)+2\eta(x)\left[\eta(y)\D w(\ff z)+\eta(z)\D w(\ff y)\right]\bigr\},
\end{array}
\end{equation}
where 
\[
\begin{array}{l}
F_1(x,y,z)=F(x,\ff y,z)+F(\ff y,x,z)-F(z,x,\ff y)+F(x,y,\ff z)-F(y,x,\ff z)+F(\ff z,x,y),
\\[4pt]
F_2(x,y,z)=\left[F(x,y,\xi)-F(\ff y,\ff x,\xi)\right]\eta(z)
    +\left[F(x,z,\xi)-F(\ff z,\ff x,\xi)\right]\eta(y)\\[4pt]
		\phantom{F_2(x,y,z)=}
    +\left[F(y,z,\xi)-F(\ff z,\ff y,\xi) \right.
		\left.+F(z,y,\xi)-F(\ff y,\ff z,\xi)\right]\eta(x),
\end{array}
\]	
\[
\begin{array}{l}
\chi_1(z)=\cosh(2v)\left[\D u(\ff z)-\D v(z)\right]+\sinh(2v)\left[\D v(\ff z)-\D u(z)\right], \\[4pt]
\chi_2(z)=\cosh(2v)\left[\D v(\ff z)-\D u(z)\right]+\sinh(2v)\left[\D u(\ff z)-\D v(z)\right].
\end{array}
\]

\end{lem}
\begin{proof}[Proof]
The Koszul equality \eqref{koszul} for the Levi-Civita connection $\overline{\n}$ of $\overline{g}$,
\eqref{F=nfi}, \eqref{F-prop}, \eqref{Fxieta}, \eqref{einm} and \eqref{cct} yield
\begin{equation}\label{gbar}
\begin{array}{l}
    2\overline{g}\left(\overline{\n}_xy,z\right)=2\exp(2u)\bigl\{\cosh(2v)\,g(\n_xy,z)
    +\sinh(2v) \bigl[g(\n_xy,\ff z)+F_3(x,y,z)\bigr]\\[4pt]
    \phantom{2\overline{g}\left(\overline{\n}_xy,z\right)=2\exp(2u)\bigl\{}
    +\psi_1(x)g(\ff y,\ff z)+\psi_1(y)g(\ff x,\ff z)-\psi_1(z)g(\ff x,\ff y)
\\[4pt]
    \phantom{2\overline{g}\left(\overline{\n}_xy,z\right)=2\exp(2u)\bigl\{}
    + \psi_2(x)g(y,\ff z)+\psi_2(y)g(x,\ff z)-\psi_2(z)g(x,\ff y)\bigr\}
\\[4pt]
    \phantom{2\overline{g}\left(\overline{\n}_xy,z\right)=}
    +\left\{\exp(2w)-\exp(2u)\cosh(2v)\right\}\bigl\{2\eta(\n_xy)\eta(z)+F_4(x,y,z)\bigr\}
\\[4pt]
    \phantom{2\overline{g}\left(\overline{\n}_xy,z\right)=}
        +2\exp(2w)\left\{\eta(y)\eta(z)\D w(x)+\eta(x)\eta(z)\D w(y)-\eta(x)\eta(y)\D w(z)\right\},
\end{array}
\end{equation}
where $\psi_1=\cosh(2v)\D u+\sinh(2v)\D v$, $
\psi_2=\cosh(2v)\D v+\sinh(2v)\D u$,
\[
\begin{array}{l}
F_3(x,y,z)=\frac12\left\{F(x,y,z)+F(y,x,z)-F(z,x,y)\right\},\\[4pt]
F_4(x,y,z)=\left[F(z,\ff y,\xi)-F(y,\ff z,\xi)\right]\eta(x)+\left[F(z,\ff x,\xi)-F(x,\ff z,\xi)\right]\eta(y)\\[4pt]
\phantom{F_3(x,y,z)=}
-\left[F(x,\ff y,\xi)+F(y,\ff x,\xi)\right]\eta(z).
\end{array}
\]

The form of \eqref{ff} follows from \eqref{F=nfi} and \eqref{gbar}.
\end{proof}

When we substitute \eqref{cct} into \eqref{defsl}, we obtain the para-Sasaki-like condition for the metric $\overline g$ as follows
\begin{equation}\label{defbar}
\begin{array}{l}
\overline F(x,y,z)=-\exp(w+2u)\bigl\{\cosh(2v)\left[\eta(z)g(\ff x,\ff y)+\eta(y)g(\ff x,\ff z)\right]\\[4pt]
\phantom{\overline F(x,y,z)=-\exp(w+2u)\bigl\{}
+\sinh(2v)\left[\eta(z)g(x,\ff y)+\eta(y)g(x,\ff z)\right]\bigr\}.
\end{array}
\end{equation}
Now we substitute \eqref{defsl} into \eqref{ff} to get
\begin{equation}\label{ff1}
\begin{array}{l}
   \overline{F}(x,y,z)=\exp(2w)\eta(x)\left\{\eta(y)\D w(\ff z)+\eta(z)\D w(\ff y)\right\}
\\[4pt]
    \phantom{\overline{F}(x,y,z)=}
    -\exp(2u)\Bigl\{
    \bigl[\cosh(2v)\eta(z)+\chi_1(z)\bigr]g(\ff x,\ff y)+\bigl[\cosh(2v)\eta(y)+\chi_1(y)\bigr]g(\ff x,\ff z)
\\[4pt]
    \phantom{\overline{F}(x,y,z)=}
    \phantom{-\exp(2u)\Bigl\{\ }%
    +\bigl[\sinh(2v)\eta(z)+\chi_2(z)\bigr]g(x,\ff y)+\bigl[\sinh(2v)\eta(y)+\chi_2(y)\bigr]g(x,\ff z)\Bigr\}.
\end{array}
\end{equation}
%
Then,  \eqref{defbar} and \eqref{ff1} imply
\begin{equation}\label{ff2}
\begin{array}{l}
    \{\exp(w)-1\}\exp(2u)\bigl\{
    \cosh(2v)\left[\eta(z)g(\ff x,\ff y)+\eta(y)g(\ff x,\ff z)\right]\\[4pt]
   \phantom{(1-\exp(w))\exp(2u)\bigl\{}
    +\sinh(2v)\left[\eta(z)g(x,\ff y)+\eta(y)g(x,\ff z)\right]\bigr\}
\\[4pt]
    +\exp(2u)\bigl\{%
    \chi_1(z) g(\ff x,\ff y)
    +\chi_1(y) g(\ff x,\ff z)
    +\chi_2 (z) g(x,\ff y)
    +\chi_2(y) g(x,\ff z)
    \bigr\}
\\[4pt]
    +\exp(2w)\eta(x)\left[\eta(y)\D w(\ff z)+\eta(z)\D w(\ff y)\right]
    =0.
\end{array}
\end{equation}
Set $x=y=\xi$ into \eqref{ff2} to get
\begin{equation}\label{www}
\D w(\ff z)=0.
\end{equation}
 Now, applying \eqref{www} we rewrite \eqref{ff2} in  the form
\begin{equation}\label{ff3}
\vartheta_1(z)g(\ff x,\ff y)+\vartheta_2(z)g(x,\ff y)+\vartheta_1(y)g(\ff x,\ff z)+\vartheta_2(y)g(x,\ff z)=0,
\end{equation}
where the 1-forms $\vartheta_1$ and $\vartheta_2$ are defined by
\begin{equation}\label{ggg}
\vartheta_1(z)=[\exp(w)-1]\cosh(2v)\eta(z)+\chi_1(z),\qquad
\vartheta_2(z)=[\exp(w)-1]\sinh(2v)\eta(z)+\chi_2(z).
\end{equation}
Taking the trace of \eqref{ff3} with respect to $x=e_i$, $z=e_i$
and $y= e_i$, $z=e_i$ to get the following system
\begin{equation}\label{ff4}
 2(n+1)\vartheta_1(z)-\eta(z)\vartheta_1(\xi)
+\vartheta_2(\ff z)=0, \qquad
    \vartheta_1(z)-\eta(z)\vartheta_1(\xi) +\vartheta_2(\ff z)=0.
\end{equation}
We obtain from \eqref{ff4} that $\vartheta_1=0$ and $\vartheta_2\circ\phi=0$.
Additionally, by the trace of  \eqref{ff3} with respect to $x=\phi e_i$, $y=e_i$
we obtain
the vanishing of $\vartheta_2$, too. 
Therefore, \eqref{ggg} imply
\begin{equation}\label{ggg1}
\chi_1(z)=[1-\exp(w)]\cosh(2v)\eta(z),\qquad
\chi_2(z)=[1-\exp(w)]\sinh(2v)\eta(z).
\end{equation}

Then, comparing \eqref{defsl} and \eqref{ff} we derive
\begin{prop}\label{prop:Sasaki}
Let $\M$ be a para-Sasaki-like Riemannian
manifold. Then the structure $(\ff,
\overline{\xi},\overline{\eta},\allowbreak{}\overline g)$ defined by
\eqref{cct} is para-Sasa\-ki-like if and only if the smooth functions
$u,v,w$ satisfy the following conditions
\begin{equation}\label{sssl}
dw\circ\ff=0,\quad \D u-\D v\circ\ff=0,\quad \D u\circ\ff-\D v=[1-\exp(w)]\eta.
\end{equation}
Consequently we have
\[
\D u(\xi)=0,\qquad \D v(\xi)=\exp(w)-1.
\]
In the case $w=0$, the global smooth functions $u$ and $v$ do not
depend on $\xi$ and they are locally defined on the 
paracomplex submanifold
$N^{2n}$, $TN^{2n}=H$. Then, the paracomplex-valued function $u+e\,v$, where $e^2=1$,
is a paraholomorphic function on $N^{2n}$. 
\end{prop}
\begin{proof}[Proof]
The equality \eqref{www} is the first part of \eqref{sssl}.
Solving the linear system \eqref{ggg1}, we obtain the second and the third
equality in \eqref{sssl}.
In the case $w=0$, we get from \eqref{sssl} the following
\[
\D u-\D v\circ\ff=0,\qquad \D u\circ\ff-\D v=0,
\]
which shows that the paracomplex function $u+e\,v$ on $N^{2n}$ is paraholomorphic.
\end{proof}

\subsection{Paracontact homothetic transformations}
We consider para\-con\-tact homothetic transformations of a para-Sasaki-like  Riemannian manifold $\M$. Since the functions $u$,
$v$, $w$ are constant, it follows from \eqref{cct} using the
Koszul formula and \eqref{gbar}
that the Levi-Civita connections
$\overline{\n}$ and $\n$ of the metrics $\overline g$ and $g$,
respectively, are related by the formula
\begin{equation}\label{barl}
\begin{array}{l}
\overline{\n}_xy=\n_xy-\exp(2u-2w)\sinh(2v)\,g(\ff x,\ff y)\xi
+\left[1-\exp(2u-2w)\cosh(2v)\right] g(x,\ff y)\xi.
\end{array}
\end{equation}

Using \eqref{barl}, we obtain the next relation between  the corresponding curvature tensors $\overline R$ and $R$
\begin{equation}\label{barRR}
\begin{array}{l}
\overline{R}(x,y)z={R}(x,y)z+\left\{1-\exp(2u-2w)\cosh(2v)\right\}\left\{g(\ff y,\ff z)\eta(x)\xi-g(\ff x,\ff z)\eta(y)\xi\right.\\[4pt]
\phantom{\overline{R}(x,y)z={R}(x,y)z+\left\{1-\exp(2u-2w)\cosh(2v)\right\}}
\left.
\,{+}\,g(y,\ff z)\ff x-g(x,\ff z)\ff y\right\}\\[4pt]
\phantom{\overline{R}(x,y)z=}
-\exp(2u-2w)\sinh(2v) %
\left\{g(y,\ff z)\eta(x)\xi-g(x,\ff z)\eta(y)\xi
+g(\ff y,\ff z)\ff x-g(\ff x,\ff z)\ff y\right\}\\[4pt]
\phantom{\overline{R}(x,y)z=}

\end{array}
\end{equation}
\begin{prop}\label{homric}
The Ricci tensor of a para-Sasaki-like Riemannian manifold
is invariant under a para\-con\-tact homothetic transformation,
\begin{equation}\label{ri}
\overline{Ric}=Ric.
\end{equation}
Moreover, we get
\begin{equation}\label{scalsas}
\begin{split}
\overline{Scal}&=\exp(-2u)\cosh(2v)Scal-\exp(-2u)\sinh(2v)Scal^*
-2n\left\{\exp(-2w)-\exp(-2u)\cosh(2v)\right\},\\[4pt]
\overline{Scal}^*&=\exp(-2u)\cosh(2v)Scal^*-\exp(-2u)\sinh(2v)Scal.
\end{split}
\end{equation}
\end{prop}
\begin{proof}[Proof]
We get  \eqref{ri} by taking the trace of \eqref{barRR}.  Consequently, the traces in \eqref{ri} imply \eqref{scalsas}.
\end{proof}
\begin {corr} Note that under  a paracontact homothetic transformation of a para-Sasaki-like Riemannian manifold the obtained space is not, in general, again para-Sasaki-like.  Indeed  the condition \eqref{sssl} is not true for constants $u,v,w\not=0$ and it is satisfied for constants $u,v,w=0$.
\end{corr}
Using Proposition~\ref{homric}  we can make Proposition~\ref{einstein}  a little bit stronger as follows
\begin{prop}
A para-Sasaki-like Riemannian  manifold $\M$ is paracontact homothetic to an Einstein para-Sasaki-like Riemannian
 manifold if and only if the underlying phpcR manifold $(N^{2n}, TN^{2n}=H,P,h)$  is an Einstein manifold with negative scalar curvature.
\end{prop}
\begin{proof}[Proof]

We consider
a paracontact homothetic trans\-for\-ma\-tion with $v=w=0$. According to Proposition~\ref{prop:Sasaki}, the manifold
 $\bigl(M,\ff,\xi,\eta,\overline
g\bigr)$, where $\overline
g=\exp(2u)\,g+\{1-\exp(2u)\}\eta\otimes\eta$ is also a para-Sasaki-like Riemannian manifold. We get the following sequence of equalities applying Proposition~\ref{homric} and \eqref{ric}
\[
\overline{Ric}^{\bar h}=\overline{Ric}{|_H}=Ric|_H=Ric^h
=\frac{Scal^h}{2n}g|_H=\frac{\exp(-2u)Scal^h}{2n}\overline{g}|_H,
\]
implying that the underlying phpcR manifold  $(N^{2n}, TN^{2n}=H,P,\bar{h})$ is  an Einstein manifold with scalar curvature  $\overline{Scal}^{\bar h}=\exp(-2u)Scal^h$.  Since $Scal^h$ is negative, we can take $u=-\frac{1}{2}\ln\left(\frac{4n^2}{-Scal^h}\right)$ to get $\overline{Scal}^{\bar h}=-4n^2$ and Proposition~\ref{einstein} shows that  $\bigl(M,\ff,\xi,\eta,\overline g\bigr)$ is an Einstein para-Sasaki-like Riemannian manifold.
\end{proof}

Suppose we have a para-Sasaki-like Riemannian
manifold which is Einstein, $Ric=-2n\,g$, and make a paracontact
homothetic transformation
$
\overline{\eta}=\eta,\quad \overline{\xi}=\xi,\quad
\overline{g}(x,y) = p\, g(x,y)+q\, g(x,\ff y)+(1-p)\eta(x)\eta(y),
$
where $p$, $q$ are constants. Using Proposition~\ref{homric} 
we obtain that
\begin{equation}\label{etaein}
\begin{split}
\overline{Ric}(x,y)=Ric(x,y)=-2n\,g(x,y)
=-\frac{2n}{p^2-q^2}\left\{p\,\overline g(x,y) - q\,\overline
g(x,\ff y) +(p^2-q^2-p)\eta(x)\eta(y)\right\}.
\end{split}
\end{equation}
We call a para-contact paracomplex 
manifold whose Ricci tensor satisfies
\eqref{etaein}  \emph{an $\eta$-paracomplex-Ein\-stein} 
and if $q=0$, we have \emph{$\eta$-Ein\-stein para-Sasaki-like Riemannian  manifold}.
Thereby, we have shown the following
\begin{prop}
Any $\eta$-paracomplex-Einstein para-Sasaki-like Riemannian  space is paracontact
ho\-mo\-the\-tic to an Ein\-stein para-Sasaki-like Riemannian  space.
\end{prop}


\end{document}